\documentclass[11pt]{article}

\usepackage{amssymb,amsmath,amsthm, calc, graphicx}
\usepackage{epsfig}
\usepackage{breqn}
\usepackage{footnpag}
\usepackage{rotating}
\usepackage{amsfonts}
\usepackage{setspace}
\usepackage{fullpage}
\usepackage{enumitem}
\usepackage{bbold} 
\usepackage{comment}
\usepackage{pgf,tikz}
\usepackage{mathrsfs}
\usetikzlibrary{arrows}
\usepackage{yhmath}
\usepackage{hyperref}

\newtheorem{thm}{Theorem}
\newtheorem{definition}{Definition}
\newtheorem{claim}[thm]{Claim}

\newcommand{\thistheoremname}{}
\newtheorem*{genericthm*}{\thistheoremname}
\newenvironment{namedthm*}[1]
{\renewcommand{\thistheoremname}{#1}%
	\begin{genericthm*}}
	{\end{genericthm*}}

\newcommand{\abs}[1]{\left\lvert{#1}\right\rvert}

\title{On the Linear Cycle Cover Conjecture of Gy\'arf\'as and S\'ark\"ozy}

\author{
	Beka Ergemlidze\thanks{
		Department of Mathematics, Central European University, Budapest.
		E-mail: beka.ergemlidze@gmail.com
	} \qquad Ervin Gy\H{o}ri \thanks{R\'enyi Institute, Hungarian Academy of Sciences and 
	Department of Mathematics, Central European University, Budapest. E-mail: gyori.ervin@renyi.mta.hu} \qquad Abhishek Methuku \thanks{Department of Mathematics, Central European University, Budapest. E-mail: abhishekmethuku@gmail.com}
}

\begin{document}

\maketitle

%
%

\begin{abstract}

A linear cycle in a $3$-uniform hypergraph $H$ is a cyclic sequence of hyperedges such that two consecutive hyperedges intersect in exactly one element and two nonconsecutive hyperedges are disjoint and $\alpha(H)$ denotes the size of a largest independent set of $H$. In this note, we show that the vertex set of every $3$-uniform hypergraph $H$ can be covered by at most $\alpha(H)$ pairwise edge-disjoint linear cycles (where we accept a vertex and a hyperedge as a linear cycle), proving a weaker version of a conjecture of Gy\'arf\'as and S\'ark\"ozy.
\end{abstract}

\section{Introduction}

A well-known theorem of P\'osa \cite{Posa} states that the vertex set of every graph $G$ can be partitioned into at most $\alpha(G)$ cycles where $\alpha(G)$ denotes the independence number of $G$ (where a vertex or an 
edge is accepted as a cycle). 

\begin{definition} 
A \emph{(linear cycle) linear path} is a (cyclic) sequence of hyperedges such that two consecutive hyperedges intersect in exactly one element and two nonconsecutive hyperedges are disjoint.

\end{definition}

An independent set of a hypergraph $H$ is a set of vertices that contain no hyperedges of $H$. Let $\alpha(H)$ denote the size of a largest independent set of $H$ and we call it the independence number of $H$. Gy\'arf\'as and S\'ark\"ozy \cite{Gy_Sar} conjectured that the following extension of P\'osa's theorem holds: One can partition every $k$-uniform hypergraph $H$ into at most $\alpha(H)$ linear cycles (here, as in P\'osa's theorem, vertices and subsets of hyperedges are accepted as linear cycles). We show the following:

\begin{thm}
\label{mainthm}
If $H$ is a $3$-uniform hypergraph, then its vertex set can be covered by at most $\alpha(H)$ edge-disjoint linear cycles (where we accept a single vertex or a hyperedge as a linear cycle).
\end{thm}


\section{Proof of Theorem \ref{mainthm}}
We call a hypergraph \emph{mixed} if it can contain hyperedges of both sizes $2$ and $3$. We will in fact prove our theorem for mixed hypergraphs (which is clearly a bigger class of hypergraphs than $3$-uniform hypergraphs). More precisely, we will prove the following stronger theorem: 

\begin{thm}
If $H$ is a mixed hypergraph, then its vertex set $V(H)$ can be covered by at most $\alpha(H)$ edge-disjoint linear cycles (where we accept a single vertex or a hyperedge as a linear cycle).
\end{thm}

\begin{proof}
We prove the theorem by induction on $\alpha(H)$. If $\abs{V(H)}=1$ or $2$ then the statement is trivial. If $\abs{V(H)} \ge 3$ and  $\alpha(H) = 1$, then $H$ contains all possible edges of size $2$ and there is a Hamiltonian cycle consisting only of edges of size $2$, which is of course a linear cycle covering $V(H)$.

Let $\alpha(H) > 1$. If $E(H) = \emptyset$, then $\alpha(H) = V(H)$ and the statement of our theorem holds trivially since we accept each vertex as a linear cycle. If $E(H) \not = \emptyset$, then let $P$ be a longest linear path in $H$ consisting of hyperedges $h_0, h_1, \ldots, h_l$ ($l \ge 0$). If $h_i$ is of size $3$, then let $h_i = v_i v_{i+1} u_{i+1}$ and if it is of size $2$, then let $h_i = v_i v_{i+1}$. A linear subpath of $P$ starting at $v_0$ (i.e., a path consisting of hyperedges $h_0, h_1, \ldots, h_j$ for some $j \le l$) is called an \emph{initial segment} of $P$. Let $C$ be a linear cycle in $H$ which contains the longest initial segment of $P$. If there is no linear cycle containing $h_0$, then we simply let $C = h_0$. 

Let us denote the subhypergraph of $H$ induced on $V(H) \setminus V(C)$ by $H \setminus C$. Let $R = \{v_ku_k \mid \{v_k,u_k\} \subseteq V(P) \setminus V(C) \text{ and } v_0 v_k u_k \in E(H)\}$ be the set of \emph{red} edges. Let us construct a new hypergraph $H'$ where $V(H') = V(H) \setminus V(C)$ and  $E(H') = E(H \setminus C) \cup R$. We will show that $\alpha(H') < \alpha(H)$ and any linear cycle cover of $H'$ can be extended to a linear cycle cover of $H$ by adding $C$ and extending the red edges by $v_0$. 

\begin{claim}
	\label{independent}
	If $I$ is an independent set in $H'$ then $I \cup {v_0}$ is an independent set in $H$.
\end{claim}
\begin{proof}
	Suppose by contradiction that $h \subseteq (I \cup {v_0})$ for some $h \in E(H)$. Then, clearly $v_0 \in h$ because otherwise $I$ is not an independent set in $H'$. Now let us consider different cases depending on the size of $h \cap (V(P) \setminus V(C))$. If $\abs{h \cap (V(P) \setminus V(C))} = 0$ then, by adding $h$ to $P$, we can produce a longer path than $P$, a contradiction. If $\abs{h \cap (V(P) \setminus V(C))} = 1$, let $h \cap (V(P) \setminus V(C)) = \{x\}$. Then the linear subpath of $P$ between $v_0$ and $x$ together with $h$ forms a linear cycle which contains a larger initial segment of $P$ than $C$, a contradiction. If $\abs{h \cap (V(P) \setminus V(C))} = 2$, then let $h \cap (V(P) \setminus V(C)) = \{x,y\}$. Let us take smallest $i$ and $j$ such that $x \in h_i$ and $y \in h_j$ (i.e., if $x \in h_i \cap h_{i+1}$ then let us take $h_i$). If $i \not = j$, say $ i < j$ without loss of generality, then the linear subpath of $P$ between $v_0$ and $x$ together with $h$ forms a linear cycle with longer initial segment of $P$ than $C$, a contradiction. Therefore, $i = j$ but in this case, $\{x,y\}$ is a red edge and so at most one of them can be contained in $I$, contradicting the assumption that $h = v_0xy \subseteq (I \cup {v_0})$. Hence, $I \cup {v_0}$ is an independent set in $H$ as desired.
\end{proof}

\begin{claim}
	\label{onered}
	The set of hyperedges of every linear cycle in $H'$ contains at most one red edge.
\end{claim}
\begin{proof}
	Suppose by contradiction that there is a linear cycle $C'$ in $H'$ containing at least two hyperedges which are red edges. Then there is a linear subpath $P'$ of $C'$ consisting of hyperedges $h'_0,h'_1,\ldots,h'_m$ such that $h'_0 := v_su_s$ and $h'_m := v_tu_t$ are red edges but $h'_k$ is not a red edge for any $1 \le k \le m-1$. Let us take the smallest $i$ such that $V(P') \cap h_i \not = \emptyset$ and then the smallest $j$ such that $h'_j \cap h_i \not = \emptyset$. It is easy to see that $\abs{V(P') \cap h_i} \le 2$ (since $i$ was smallest). If $\abs{h'_j \cap h_i} = 1$, then the linear cycle consisting of hyperedges $h'_1,\ldots,h'_j$ and $h_i, h_{i-1}, \ldots, h_0$ and $v_0v_su_s$ contains a larger initial segment of $P$ than $C$ (as $h'_j \cap h_i \in V(P) \setminus V(C)$), a contradiction. If $\abs{h'_j \cap h_i} = 2$, then notice that $\abs{h'_{j+1} \cap h_i} = 1$. Now the linear cycle consisting of the hyperedges $h'_{m-1}, h'_{m-2}, \ldots, h'_{j+1}$ and $h_i, h_{i-1}, \ldots, h_0$ and $v_0v_tu_t$ contains a larger initial segment of $P$ than $C$, a contradiction. 
\end{proof}


By Claim \ref{independent}, $\alpha(H') \le \alpha(H)-1$. So by induction hypothesis, $V(H')$ can be covered by at most  $\alpha(H)-1$ edge-disjoint linear cycles (where we accept a single vertex or a hyperedge as a linear cycle). Now let us replace each red edge $\{x, y\}$ with the hyperedge $xyv_0$ of $H$. Claim \ref{onered} ensures that in each of these linear cycles at most one of the hyperedges is a red edge. Therefore, it is easy to see that even after the above replacement linear cycles of $H'$ remain as \emph{linear} cycles in $H$ and they cover $V(H') = V(H) \setminus V(C)$. Now $C$ along with these linear cycles give us at most $\alpha(H)$ edge-disjoint linear cycles covering $V(H)$, as desired. 
\end{proof}

\section*{Acknowledgement}
The research of the second author is partially supported by the National Research, Development and Innovation Office NKFIH, grant K116769.

\end{document}